\documentclass{amsart}

\usepackage{amsmath, amsthm, amsfonts, amssymb}
\usepackage{latexsym}
\usepackage{setspace}
\usepackage[pdftex]{graphicx}
\usepackage{enumitem}
\usepackage{subfigure}
\usepackage{hyperref}
\usepackage{color}
\usepackage{pinlabel}
\usepackage{microtype}

\theoremstyle{plain}
\newtheorem{theorem}{Theorem}
\newtheorem{corollary}[theorem]{Corollary}
\newtheorem{lemma}[theorem]{Lemma}
\newtheorem{proposition}[theorem]{Proposition}
\newtheorem{conjecture}[theorem]{Conjecture}

\theoremstyle{definition}
\newtheorem{definition}[theorem]{Definition}

\newtheorem*{remark}{Remark}

\makeatletter
\g@addto@macro\@floatboxreset\centering
\makeatother

\DeclareMathOperator{\rk}{rk}
\DeclareMathOperator{\HFhat}{\widehat{HF}}

\newcommand{\numset}[1]{\mathbb{#1}}

\newcommand{\Z}{\numset{Z}}
\newcommand{\Q}{\numset{Q}}
\newcommand{\R}{\numset{R}}

\newcommand{\B}{\mathcal{B}}

\definecolor{a-color}{RGB}{255,0,165}
\definecolor{b-color}{RGB}{0,165,255}
\definecolor{c-color}{RGB}{255,89,0}
\definecolor{d-color}{RGB}{0,255,89}
\definecolor{greengen-color}{RGB}{0,255,0}
\definecolor{redgen-color}{RGB}{255,0,0}
\definecolor{bluegen-color}{RGB}{0,0,255}
\definecolor{mu-color}{RGB}{165,0,255}
\definecolor{s-color}{RGB}{255,165,0}
\definecolor{ab-color}{RGB}{133,204,0}

\title{Non-left-orderable surgeries on 1-bridge braids}

\author[S.~Liang]{Shiyu Liang}
\address{Department of Mathematics, University of Texas at Austin, Austin, Texas 78712, 
  USA
}
\email{\href{mailto: shiyu.liang@math.utexas.edu}{shiyu.liang@math.utexas.edu}}

\subjclass[2010]{57M25 (20F60 57M50)}
\keywords{left-orderability, 1-bridge braid, Dehn surgery}

\begin{document}
\begin{abstract}
Boyer, Gordon, and Watson \cite{BGWconj} have conjectured that an irreducible rational homology $3$-sphere is an L-space if and only if its fundamental group is not left-orderable. Since Dehn surgeries on knots in $S^3$ can produce large families of L-spaces, it is natural to examine the conjecture on these $3$-manifolds. Greene, Lewallen, and Vafaee \cite{GLV11k} have proved that all $1$-bridge braids are L-space knots. In this paper, we consider three families of $1$-bridge braids. First we calculate the knot groups and peripheral subgroups. We then verify the conjecture on the three cases by applying the criterion \cite{CGHVnlo} developed by Christianson, Goluboff, Hamann, and Varadaraj, when they verified the same conjecture for certain twisted torus knots and generalized the criteria in \cite{CWcri} and \cite{ITcri}.
\end{abstract}
\maketitle

\section{Introduction}
Let $Y$ be a rational homology sphere, and denote by $\HFhat(Y)$ the ``hat" version of Heegaard Floer homology, as defined in \cite{OSHF}. The following result is shown by Ozsv\'{a}th and Szab\'{o} in \cite{OSHF2}: $\rk\widehat{HF}(Y)\ge \lvert H_{1}(Y; \Z) \rvert$. As a space with minimal Heegaard Floer homology, an L-space is defined as follows:
\begin{definition}
A closed, connected, orientable $3$-manifold $Y$ is an \emph{L-space} if it is a rational homology sphere with the property
\[\rk\widehat{HF}(Y) = \lvert H_{1}(Y; \Z) \rvert.\]
\end{definition}
It is interesting that L-spaces might be characterized by properties of their fundamental groups, which seem to be unrelated to Heegaard Floer homology. Recall the following definition:
\begin{definition}
A non-trivial group $G$ is called \emph{left-orderable} if there exists a strict total ordering $<$ on $G$ which is left-invariant, i.e., 
\[g<h \Rightarrow fg<fh,\forall f, g, h \in G.\]
\end{definition}
The identity element is always denoted by symbol $1$ in this paper, and the symbols $>$, $\le$, and $\ge$ have the usual meaning.

In \cite{BGWconj}, Boyer, Gordon, and Watson made the following conjecture that indicates a connection between L-spaces and left-orderability of their fundamental groups.
\begin{conjecture}
[{\cite[Conjecture 3]{BGWconj}}]
\label{conj:BGW}
An irreducible rational homology $3$-sphere is an L-space if and only if its fundamental group is not left-orderable.
\end{conjecture}
Since Dehn surgeries on knots in $S^{3}$ provide large families of $3$-manifolds, it is natural to consider if the conjecture can be verified on them. The concept of an L-space knot is needed to simplify our discussion:
\begin{definition}
An knot $K$ in $S^{3}$ is an \emph{L-space knot} if it admits some non-trivial Dehn surgery yielding an L-space.
\end{definition}
Conjecture~\ref{conj:BGW} has been verified for certain families of Dehn surgeries. For instance, it has been verified for sufficiently large surgery on some twisted torus knots in \cite[Theorem 14]{CGHVnlo}. Our goal of this paper is to verify the conjecture on another similar family of knots. The specific family of knots that will be worked on is $1$-bridge braids, which are first studied by Berge and Gabai in \cite{BerLSK} and \cite{Gab1bb}, and are a natural subset of a broad family of $(1,1)$-knots that are L-space knots (\cite{GLV11k}). They are defined as follows:

\begin{figure}
\centering
\includegraphics[scale = 1]{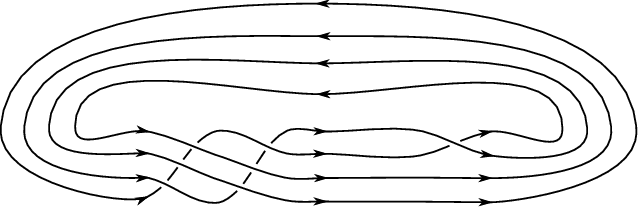}
\caption{The $(4,2,1)$-$1$-bridge braid, denoted by $B(4, 2, 1)$.}
\label{fig:onebb}
\end{figure}

\begin{definition}
[{\cite[Definition 1.3]{GLV11k}}]
A knot in the solid torus $D^{2} \times S^{1}$ is a \emph{$1$-bridge braid} if it is isotopic to a union of two arcs $\rho \cup \tau$ such that
\begin{enumerate}[label=(\roman*)]
\item $\rho \subset \partial(D^{2} \times S^{1})$ is braided, i.e., transverse to each meridian $\partial D^{2} \times pt$, and

\item $\tau$ is a bridge, i.e., properly embedded in some meridional disk $D^{2} \times pt$.
\end{enumerate}
It is positive if $\rho$ is a positive braid in the usual sense. A knot in $S^3$ is a \emph{$1$-bridge braid} if it is isotopic to a $1$-bridge braid supported in a solid torus coming from a genus-$1$ Heegaard splitting of $S^3$.
\end{definition}
To present a $1$-bridge braid, let
\[\B_{\omega} = \langle \sigma_1, \sigma_2, \cdots, \sigma_{\omega-1} | \sigma_i \sigma_{i+1} \sigma_i = \sigma_{i+1} \sigma_i \sigma_{i+1}, 1 \le i \le \omega - 2,\]
\[ \sigma_j \sigma_k = \sigma_k \sigma_j, 1\le j, k \le \omega - 1, \lvert j - k \rvert \ge 2 \rangle\]
denote the braid group with $\omega$ strands, where $\sigma_{i}$ gives strands $i$, $i+1$ a right-hand half twist, and the braids are composed from right to left. It follows from Gabai's classification theorem of $1$-bridge braids in the solid torus (\cite[Proposition 2.3]{Gab1bb}) that every $1$-bridge braid in $S^3$ is of the form $B(\omega, t + m\omega, b) = \overline{\B(\omega, t + m\omega, b)}$, i.e., the braid closure of
\[\B(\omega, t + m\omega, b) = (\sigma_{1}\sigma_{2}\cdots\sigma_{b})(\sigma_{1}\sigma_{2}\cdots\sigma_{\omega-1})^{t + m\omega}\in \B_{\omega},\]
\[1\le b\le \omega-2, 1\le t\le \omega-2, m\in\Z.\]
We will be interested in the case that $B(\omega, t + m\omega, b)$ is a knot rather than a link.

J. and S. Rasmussen conjectured that every $1$-bridge braid is an L-space knot at the end of \cite{RRconj}, and then Greene, Lewallen, and Vafaee proved a more general result, of which the following is a special case:
\begin{theorem}
[{\cite[Theorem 1.4]{GLV11k}}]
\label{thm:OBBRLSK}
All $1$-bridge braids in $S^3$ are L-space knots.
\end{theorem}
In particular, positive $1$-bridge braids admit positive L-space surgeries. We will verify Conjecture~\ref{conj:BGW} on three families of $1$-bridge braids. We first claim that they are all knots, and the proof will be given in Section~\ref{sec:Knots}.
\begin{proposition}

$B(\omega, t + m\omega, b)$ is a knot if
\begin{enumerate}[label=(\roman*)]
\label{prop:knots}
\item
\label{itm:knot1}
\[t = 1, b = 2k, 1\le k\le [(\omega-2)/2],\]
\item
\label{itm:knot2}
\[\omega = 2n + 1, t = 2n - 1, b = 2k, 1\le k\le n - 1,\text{or}\]
\item
\label{itm:knot3}
\[\omega = 2n, t = 2n - 2, b = 2k - 1, 1\le k\le n - 1.\]
\end{enumerate}
\end{proposition}
The following result due to Ozsv\'ath and Szab\'o completely characterizes which surgeries on an L-space knot yield L-spaces.
\begin{theorem}
[{\cite{OS3}}]
Let $K\subset S^{3}$, and suppose that there exists $\frac{p}{q}\in\Q_{\ge 0}$ such that $S^{3}_{p/q}(K)$ is an L-space. Then $S^{3}_{p'/q'}(K)$ is an L-space if and only if $\frac{p'}{q'}\ge 2g(K)-1$.
\end{theorem}
In the theorem, $S^{3}_{p/q}(K)$ denotes $p/q$-surgery on knot $K$ in $S^{3}$, and $g(K)$ is the genus of the knot. In order to provide a sufficient condition on the knot group of $K$ in $S^{3}$ to imply that $r$-surgery on $K$ yields a manifold with non-left-orderable fundamental group, the following equivalent condition for left-orderability is required.
\begin{theorem}
[{\cite[Theorem~6.8]{GhyRLO}}]
Let $G$ be a countable group. Then the following are equivalent:

\begin{enumerate}[label=(\roman*)]
	\item $G$ acts faithfully on the real line by order-preserving homeomorphisms.
	\item $G$ is left-orderable.
	\end{enumerate}
\end{theorem}
With this theorem, Christianson, Goluboff, Hamann, and Varadaraj generalized Ichihara and Temma's result in \cite{ITcri} and established the following criterion for non-left-orderability:
\begin{theorem}
[{\cite[Theorem~10]{CGHVnlo}}]
\label{thm:CGHVCond}
Let $K$ be a non-trivial knot in $S^{3}$. Let $G$ denote the knot group of $K$, and let $G(p/q)$ be the quotient of $G$ resulting from $p/q$-surgery. Let $\mu$ be a meridian of $K$ and $s$ be a $v$-framed longitude with $v > 0$. Suppose that $G$ has two generators, $x$ and $y$, such that $x = \mu$ and $s$ is a word which excludes $x^{-1}$ and $y^{-1}$ and contains at least one $x$. Suppose further that every homomorphism $\Phi:G(p/q)\rightarrow Homeo^{+}(\R)$ satisfies $\Phi(x)t > t$ for all $t\in\R$ implies $\Phi(y)t \ge t$ for all $t\in\R$. If $p,q>0$, then, for $p/q\ge v$, $G(p/q)$ is not left-orderable.
\end{theorem}
We will use similar notations to \cite{CGHVnlo} and work with the usual order on $\R$. Denote by $M(\omega, t + m\omega, b, r)$ the $3$-manifold produced by $r$-surgery on $B(\omega, t + m\omega, b)$, and let $G(\omega, t + m\omega, b) = \pi_{1}(S^3\setminus B(\omega, t + m\omega, b))$, $G(\omega, t + m\omega, b, r)=\pi_{1}(M(\omega, t + m\omega, b, r))$.
Throughout, we will assume $\omega \ge 3$, $1\le t\le\omega - 2$, $1\le b\le\omega - 2$, $m \ge 0$, $r\in\Q_{\ge 0}$.

In this paper, we apply Theorem~\ref{thm:CGHVCond} to prove the following result.
\begin{theorem}
\label{thm:Lia}
For sufficiently large $r$, the L-spaces produced by $r$-surgery on $B(\omega, t + m\omega, b)$, where $\omega$, $t$, and $b$ satisfy the conditions in Proposition \ref{prop:knots}, and $m\ge 0$, all have non-left-orderable fundamental groups, i.e., $G(\omega, 1 + m\omega, 2k, r)$, $G(2n + 1, 2n - 1 + m(2n + 1), 2k, r)$, and $G(2n,2n - 2 +  2mn, 2k-1, r)$ are non-left-orderable if $r \ge \omega+2k-1$, $r \ge 2n[2n - 1 + m(2n + 1)] + 2k$, and $r \ge (2n-1)(2n-2+2mn) + 2k-1$ respectively.
\end{theorem}
\noindent\textbf{Acknowledgements.} We thank Jennifer Hom and Sudipta Kolay for providing the resources and background material necessary for our work and for their guidance throughout. We also thank the School of Mathematics of Georgia Institute of Technology for their support and giving us the opportunity to pursue this research project. This REU program was partially funded by NSF grant DMS-1552285.

\section{Proving All the Three Families of 1-bridge Braids are Knots}
\label{sec:Knots}
First, we define a surjective group homomorphism $h: \B_{\omega} \rightarrow S_{\omega}$ by mapping every braid $\sigma_i$ to $s_i$ in symmetric group
\[S_{\omega} = \langle s_1, s_2, \cdots, s_{\omega-1} | s_i s_{i+1} s_i = s_{i+1} s_i s_{i+1}, 1 \le i \le \omega - 2,\]
\[ s_j s_k = s_k s_j, s_j^2 = 1, 1\le j, k\le\omega - 1, \lvert j - k\rvert\ge 2 \rangle.\]
For any braid $\sigma \in \B_{\omega}$, $h(\sigma)$ is called \emph{the permutation induced by $\sigma$}, so $\sigma$ acts on the $\omega$ endpoints of the braid by permutation $h(\sigma)$. The permutations are composed from right to left. Denote $h(\B(\omega, t + m\omega, b))(i)$ briefly by $\B(\omega, t + m\omega, b)(i)$.

In this section, we prove Proposition~\ref{prop:knots}.

\begin{figure}
\centering
\includegraphics[scale=1]{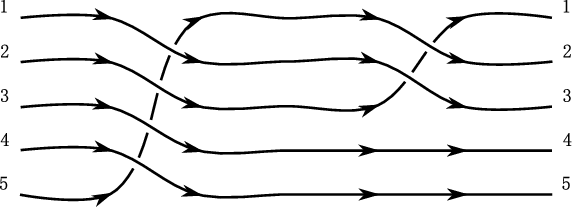}
\caption{The braid $\B(5, 1, 2)$, which induces the permutation $(1, 3, 4, 5, 2)$.}
\label{fig: Braid_5_1_2}
\end{figure}

\begin{proof}[Proof of Proposition~\ref{prop:knots}\ref{itm:knot1}]
We claim that $\B(\omega, 1 + m\omega, 2k)$ induces the action on the endpoints by the permutation
\[p=(1, 3, 5, \cdots, 2k-1, 2k+1, 2k+2, 2k+3, \cdots, \omega, 2, 4, 6,\cdots, 2k),\]
by forgetting how the strands twist and cross, i.e., if strand $i$ is sent to be strand $j$ by $\B(\omega, 1 + m\omega, 2k)$, then $p(i)=j$ (see Figure~\ref{fig: Braid_5_1_2}).

In fact, $\B(\omega, 1 + m\omega, 2k)$ has the expression:
\[\B(\omega, 1 + m\omega, 2k)=(\sigma_{1}\sigma_{2}\cdots\sigma_{2k})(\sigma_{1}\sigma_{2}\cdots\sigma_{\omega-1})^{1+m\omega}.\]
If $1\le i\le 2k-1$ is odd, $(\sigma_{1}\sigma_{2}\cdots\sigma_{\omega-1})^{1+m\omega}(i) = i+1\le 2k$, and $(\sigma_{1}\sigma_{2}\cdots\sigma_{2k})(i+1) = i+2$, so $\B(\omega, 1 + m\omega, 2k)(i) = i+2 = p(i)$.

Similarly, if $2k+1\le i\le\omega-1$, $(\sigma_{1}\sigma_{2}\cdots\sigma_{\omega-1})^{1+m\omega}(i) = i+1\ge 2k + 2 > 2k + 1$, so $(\sigma_{1}\sigma_{2}\cdots\sigma_{2k})(i+1) = i+1$. Thus, $\B(\omega, 1 + m\omega, 2k)(i) = i+1 = p(i)$. 

For strand $\omega$, $(\sigma_{1}\sigma_{2}\cdots\sigma_{\omega-1})^{1+m\omega}(\omega) = 1$, and then $(\sigma_{1}\sigma_{2}\cdots\sigma_{2k})(1) = 2$, so $\B(\omega, 1 + m\omega, 2k)(\omega) = 2 = p(\omega)$. 

If $2\le i\le 2k-2$ is even, $(\sigma_{1}\sigma_{2}\cdots\sigma_{\omega-1})^{1+m\omega}(i) = i+1\le 2k-1$, and $(\sigma_{1}\sigma_{2}\cdots\sigma_{2k})(i+1) = i+2$, so $\B(\omega, 1 + m\omega, 2k)(i) = i+2 = p(i)$. 

Finally, for strand $2k$, $(\sigma_{1}\sigma_{2}\cdots\sigma_{\omega-1})^{1+m\omega}(2k) = 2k+1$, and $(\sigma_{1}\sigma_{2}\cdots\sigma_{2k})(2k+1) = 1$, so $\B(\omega, 1 + m\omega, 2k)(2k) = 1 = p(2k)$.

Therefore, $B(\omega, 1 + m\omega, 2k)$ is a knot.
\end{proof}
\begin{figure}
\centering
\includegraphics[scale=1]{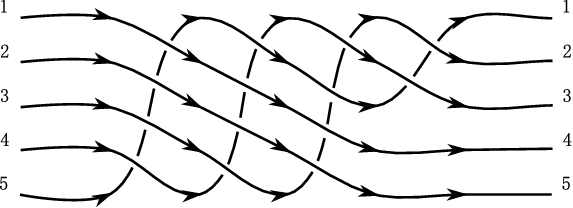}
\caption{The braid $\B(5, 3, 2)$.}
\label{fig: Braid_5_3_2}
\end{figure}
\begin{proof}[Proof of Proposition~\ref{prop:knots}\ref{itm:knot2}]
We claim that $\B(2n + 1, 2n - 1 + m(2n + 1), 2k)$ induces the action on the endpoints by the permutation
\[p = (1, 2n, 2n-2, 2n-4, \cdots, 2k+2, 2k+1, 2k, \cdots, 2, 2n+1, 2n-1, 2n-3, \cdots, 2k+3),\]
as shown in Figure~\ref{fig: Braid_5_3_2}.

As in the proof of Proposition~\ref{prop:knots}\ref{itm:knot1}, we achieve it by dealing with each strand in the order appeared in the expression of permutation $p$ we presented above. In fact, $\B(2n + 1, 2n - 1 + m(2n + 1), 2k)$ has the expression:
\[\B(2n + 1, 2n - 1 + m(2n + 1), 2k)=(\sigma_{1}\sigma_{2}\cdots\sigma_{2k})(\sigma_{1}\sigma_{2}\cdots\sigma_{2n})^{2n - 1 + m(2n + 1)}.\]
$(\sigma_{1}\sigma_{2}\cdots\sigma_{2n})^{2n - 1 + m(2n + 1)}(1) = 2n$. Then, since $n\ge k+1$, $2n\ge 2k+2 > 2k+1$, $(\sigma_{1}\sigma_{2}\cdots\sigma_{2k})(2n) = 2n$. Thus, $\B(2n + 1, 2n - 1 + m(2n + 1), 2k)(1) = 2n = p(1)$.

For strand $2k+4\le i\le 2n$ an even number, $(\sigma_{1}\sigma_{2}\cdots\sigma_{2n})^{2n - 1 + m(2n + 1)}(i)\equiv i+2n-1(mod\ 2n + 1)$. Since $1 < 2k+2\le i+2n-1-\omega = i-2 \le 2n-2 < \omega$, $(\sigma_{1}\sigma_{2}\cdots\sigma_{2n})^{2n - 1 + m(2n + 1)}(i) = i-2$. Then, since $i-2\ge 2k + 2 > 2k + 1$, $(\sigma_{1}\sigma_{2}\cdots\sigma_{2k})(i-2) = i-2$. Therefore, $\B(2n + 1, 2n - 1 + m(2n + 1), 2k)(i) = i-2 = p(i)$.

For strand $3 \le i\le 2k+2$, $(\sigma_{1}\sigma_{2}\cdots\sigma_{2n})^{2n - 1 + m(2n + 1)}(i)\equiv i+2n-1(mod\ 2n + 1)$. Since $1\le i+2n-1-\omega = i-2\le 2k < \omega$, $(\sigma_{1}\sigma_{2}\cdots\sigma_{2n})^{2n - 1 + m(2n + 1)}(i) = i-2$. Then, since $1\le i-2\le b$, $(\sigma_{1}\sigma_{2}\cdots\sigma_{2k})(i-2) = i-1$. Therefore, $\B(2n + 1, 2n - 1 + m(2n + 1), 2k)(i) = i-1 = p(i)$.

For strand $2$, $(\sigma_{1}\sigma_{2}\cdots\sigma_{2n})^{2n - 1 + m(2n + 1)}(2) = 2n+1$, and $(\sigma_{1}\sigma_{2}\cdots\sigma_{2k})(2n+1) = 2n+1$, so $\B(2n + 1, 2n - 1 + m(2n + 1), 2k)(2) = 2n+1 = p(2)$.

For strand $2k+5\le i\le 2n+1$ an odd number, $(\sigma_{1}\sigma_{2}\cdots\sigma_{2n})^{2n - 1 + m(2n + 1)}(i)\equiv i+2n-1(mod\ 2n + 1)$. Since $1\le 2k+3\le i+2n-1-\omega = i-2 \le 2n-1 < \omega$, $(\sigma_{1}\sigma_{2}\cdots\sigma_{2n})^{2n - 1 + m(2n + 1)}(i) = i-2$. Then, since $i-2\ge 2k + 3 > 2k + 1$, $(\sigma_{1}\sigma_{2}\cdots\sigma_{2k})(i-2) = i-2$. Therefore, $\B(2n + 1, 2n - 1 + m(2n + 1), 2k)(i) = i-2 = p(i)$.

Finally, for strand $2k+3$, we have $(\sigma_{1}\sigma_{2}\cdots\sigma_{2n})^{2n - 1 + m(2n + 1)}(2k+3) \equiv 2k + 2n + 2(mod\ 2n + 1)$. Then, since $1 \le 2k + 2n + 2 - \omega = 2k + 1 \le 2n - 1 < \omega$, $(\sigma_{1}\sigma_{2}\cdots\sigma_{2n})^{2n - 1 + m(2n + 1)}(2k+3) = 2k + 1$. Since $(\sigma_{1}\sigma_{2}\cdots\sigma_{2k})(2k + 1) = 1$, $\B(2n + 1, 2n - 1 + m(2n + 1), 2k)(2k+3) = 1 = p(2k+3)$.

Therefore, $B(2n + 1, 2n - 1 + m(2n + 1), 2k)$ is a knot.
\end{proof}
\begin{figure}
\centering
\includegraphics[scale=1]{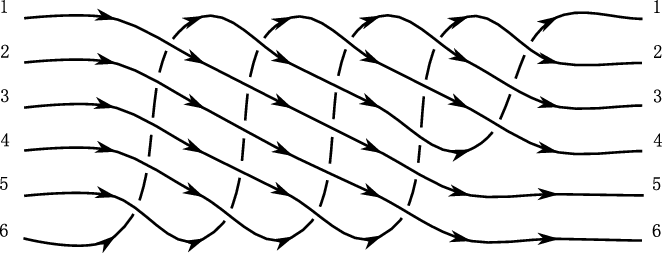}
\caption{The braid $\B(6, 4, 3)$.}
\label{fig: Braid_6_4_3}
\end{figure}
\begin{proof}[Proof of Proposition~\ref{prop:knots}\ref{itm:knot3}]
We claim that $\B(2n, 2n - 2 + 2mn, 2k - 1)$ induces the action on the endpoints by the permutation
\[p=(1, 2n-1, 2n-3, 2n-5, \cdots, 2k+1, 2k, 2k - 1, \cdots, 2, 2n, 2n-2, 2n-4, \cdots, 2k+2),\]
as shown in Figure~\ref{fig: Braid_6_4_3}.

Again, we prove it in the same order as shown in the expression of $p$. In fact, since $\B(2n, 2n - 2 + 2mn, 2k - 1)$ has the expression:
\[\B(2n, 2n - 2 + 2mn, 2k - 1)=(\sigma_{1}\sigma_{2}\cdots\sigma_{2k-1})(\sigma_{1}\sigma_{2}\cdots\sigma_{2n - 1})^{2n - 2 + 2mn},\]
$(\sigma_{1}\sigma_{2}\cdots\sigma_{2n - 1})^{2n - 2 + 2mn}(1) = 2n+1$. Then, since $2n - 1\ge 2k+1 > 2k$, $(\sigma_{1}\sigma_{2}\cdots\sigma_{2k-1})(2n+1) = 2n+1$. Thus, $\B(2n, 2n - 2 + 2mn, 2k - 1)(1) = 2n+1 = p(1).$

For strand $2k+3\le i\le 2n-1$ an odd number, $(\sigma_{1}\sigma_{2}\cdots\sigma_{2n-1})^{2n - 2 + 2mn}(i)\equiv i+2n-2(mod\ 2n)$. Since $1 < 2k+1\le i+2n-2-\omega = i-2 \le 2n-3 < \omega$, $(\sigma_{1}\sigma_{2}\cdots\sigma_{2n-1})^{2n-2}(i) = i-2$. Then, since $i-2\ge 2k+1 > 2k$, $(\sigma_{1}\sigma_{2}\cdots\sigma_{2k-1})(i-2) = i-2$. Therefore, $\B(2n, 2n - 2 + 2mn, 2k - 1)(i) = i-2 = p(i).$

The remainder of the proof of the claim follows analogously to the proof of Proposition~\ref{prop:knots}\ref{itm:knot2}.

Therefore, $B(2n, 2n - 2 + 2mn, 2k - 1)$ is a knot.
\end{proof}
\section{Computing Knot Groups and Peripheral Subgroups}
\label{sec:Groups}
It is a fact that (see \cite{LicGTM}), for a non-trivial knot, the fundamental group of the boundary of the knot complement injects into the knot group, and its image up to conjugation is the peripheral subgroup, so all peripheral subgroups are abelian.
We first derive the knot groups of $B(\omega, 1 + m\omega, 2k)$, $B(2n + 1, 2n - 1 + m(2n + 1), 2k)$, and $B(2n,2n - 2 +  2mn, 2k-1)$. We will achieve this for all three cases by the same way both Clay and Watson in \cite{CWcri} and Christianson et al. in \cite{CGHVnlo} did: by finding a genus-$2$ Heegaard splitting of $S^3$ with the knot embedded on the Heegaard surface, we can apply Seifert-van Kampen Theorem on the two handlebodies of genus $2$ (see Figures~\ref{fig: Hmtpy_Knot_Complement_5_1_2}, \ref{fig: Hmtpy_Knot_Complement_5_3_2}, and \ref{fig: Hmtpy_Knot_Complement_6_4_3}).

\begin{figure}
\centering
\includegraphics[scale=1]{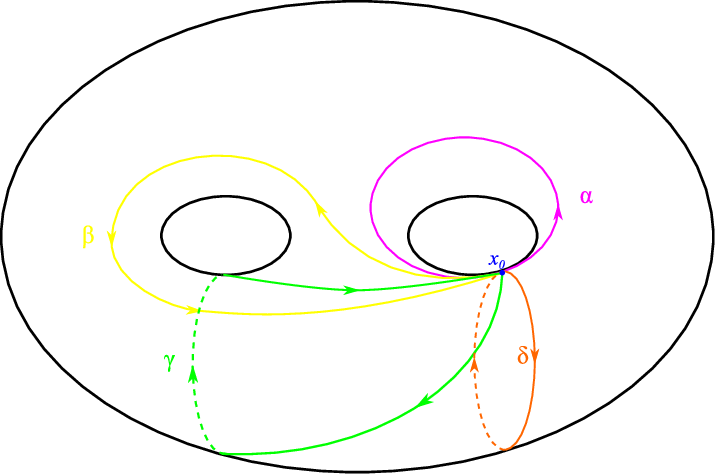}
\caption{Basepoint $x_0$ and generators $\alpha$, $\beta$, $\delta$, $\gamma$ for the fundamental group of the Heegaard surface $\Sigma$.}
\label{fig:Generators_for_H0_and_H1}
\end{figure}

\begin{figure}
\centering
\includegraphics[scale=1]{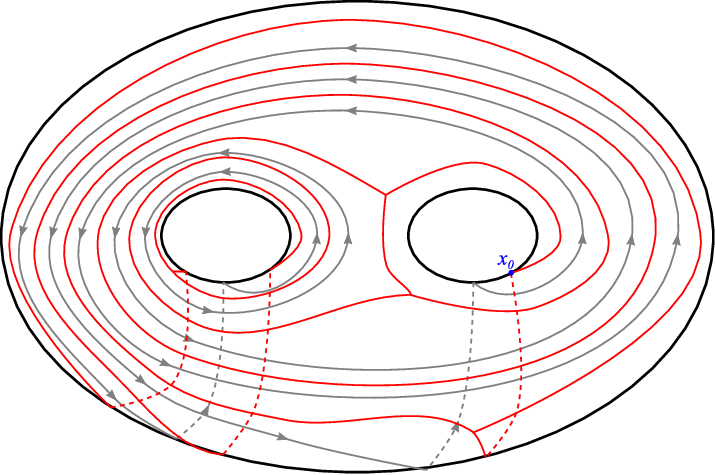}
\caption{The red graph is homotopy equivalent to the complement of $B(\omega, 1 + m\omega, 2k)$ on the Heegaard surface $\Sigma$. Here, we show $B(5, 1, 2)$ as an example. The knot is shown in gray.}
\label{fig: Hmtpy_Knot_Complement_5_1_2}
\end{figure}

\begin{proposition}
\label{prop: 1}
For $B(\omega, 1 + m\omega, 2k)$, \\
(a) The knot group is
\begin{equation}
	\nonumber	
	\resizebox{\hsize}{!}{$G(\omega, 1 + m\omega, 2k)=\langle \alpha,\beta|(\beta\alpha)^{k}\beta = \{(\beta\alpha)^k\beta^{\omega-2k}[(\beta\alpha)^k\beta\alpha^{-1}(\alpha\beta)^{-k}]^{-\omega+k}\}^m(\alpha\beta)^{k}\alpha\rangle.$}
\end{equation}
(b) The peripheral subgroup is generated by the meridian
\[\mu =\alpha\]
and the surface framing
\[s=\mu^{(\omega-1)(1+m\omega)+2k}\lambda=\alpha(\beta\alpha)^{k}\beta^{\omega-2k}(\alpha\beta)^{k}.\]
\begin{proof}
Let $S^3 = U \cup_\Sigma V$ be the genus-$2$ Heegaard splitting of $S^3$ with the knot $B(\omega, 1 + m\omega, 2k)$ embedded on the Heegaard surface, as shown in Figure~\ref{fig: Hmtpy_Knot_Complement_5_1_2}. Then $\pi_1(U)$ is the free group on the generators $\alpha$ and $\beta$, and $\pi_1(V)$ is the free group on the generators $\delta$ and $\gamma$ (see Figure~\ref{fig:Generators_for_H0_and_H1}). Using Seifert-van Kampen Theorem, we can then express $G(\omega, 1 + m\omega, 2k)$ as a free product with amalgamation of $\pi_1(U)$ and $\pi_1(V)$. In order to figure out the amalgamation, we need the images of the generators of $\pi_1(\Sigma \setminus \nu (B(\omega, 1 + m\omega, 2k)))$ under the map induced by inclusion into $\pi_1(U)$ and $\pi_1(V)$.
\begin{figure}
\centering
\includegraphics[scale=1]{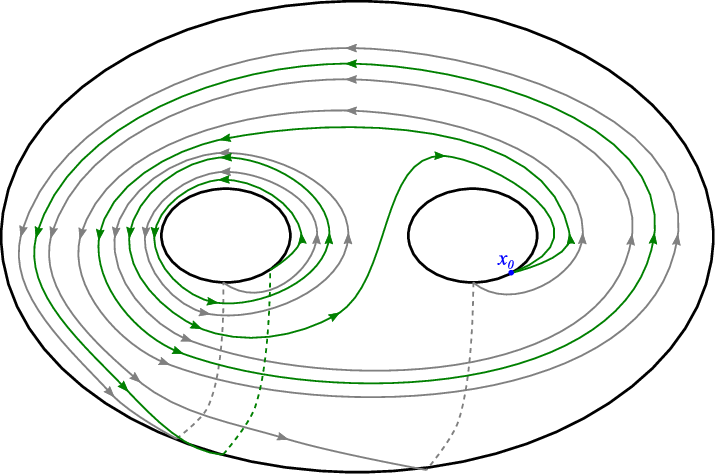}
\caption{The first generator for the fundamental group of $\Sigma \setminus \nu(B(\omega, 1 + m\omega, 2k))$.}
\label{fig: Generator1_Hmtpy_Knot_Complement_5_1_2}
\end{figure}
\begin{figure}
\centering
\includegraphics[scale=1]{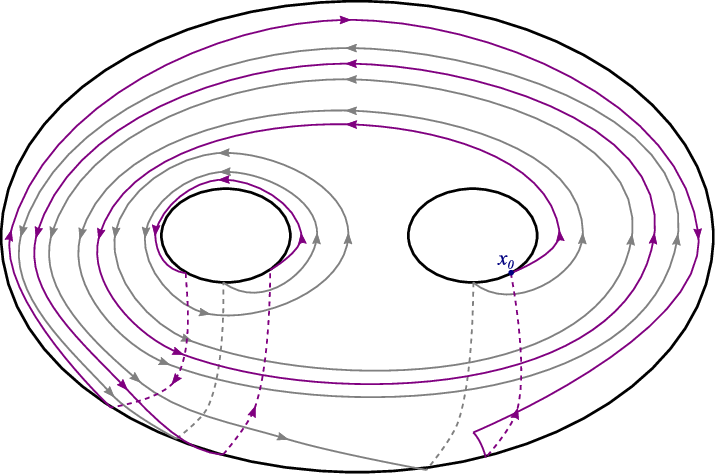}
\caption{The second generator for the fundamental group of $\Sigma \setminus \nu(B(\omega, 1 + m\omega, 2k))$.}
\label{fig: Generator2_Hmtpy_Knot_Complement_5_1_2}
\end{figure}

Now, $\Sigma \setminus \nu (B(\omega, 1 + m\omega, 2k))$ is homotopy equivalent to a twice-punctured genus-1 surface whose fundamental group is generated by the green, purple, and olive loops in Figures~\ref{fig: Generator1_Hmtpy_Knot_Complement_5_1_2}, \ref{fig: Generator2_Hmtpy_Knot_Complement_5_1_2}, and \ref{fig: Generator3_Hmtpy_Knot_Complement_5_1_2}. The green loop has image $\alpha(\beta\alpha)^{k}\beta^{\omega-2k}\alpha^{-1}$ in $\pi_1(U)$ and image $\gamma^{m(\omega-k)+1}$ in $\pi_1(V)$, so we get the relation
\[\alpha(\beta\alpha)^{k}\beta^{\omega-2k}\alpha^{-1} = \gamma^{m(\omega-k)+1}.\]
Likewise, the purple loop gives 
\[\alpha(\beta\alpha)^{k}\beta\alpha^{-1} = \gamma^{m(k+1)}\delta,\]
and the olive loop gives
\[\alpha(\alpha\beta)^{k} = \gamma^{mk}\delta.\]
The second and the third relation give us the expression for $\gamma^m$ in terms of $\alpha$ and $\beta$:
\[\gamma^{m} = \alpha(\beta\alpha)^{k}\beta\alpha^{-1}(\alpha\beta)^{-k}\alpha^{-1}.\]
Then by the first relation, replacing $\gamma^{m(\omega-k)}$ of the right side, we can get a expression for $\gamma$ in terms of $\alpha$ and $\beta$:
\[\gamma = \alpha(\beta\alpha)^k\beta^{\omega-2k}\alpha^{-1}\gamma^{-m(\omega-k)} = \alpha(\beta\alpha)^k\beta^{\omega-2k}\alpha^{-1}[\alpha(\beta\alpha)^k\beta\alpha^{-1}(\alpha\beta)^{-k}\alpha^{-1}]^{-\omega+k}\]
\[= \alpha(\beta\alpha)^k\beta^{\omega-2k}[(\beta\alpha)^k\beta\alpha^{-1}(\alpha\beta)^{-k}]^{-\omega+k}\alpha^{-1}.\]
Also, from the third relation, we can write $\delta$ in terms of $\alpha$ and $\beta$:
\begin{equation}
	\nonumber
	\begin{aligned}
		\delta & = \gamma^{-mk}\alpha(\alpha\beta)^k = [\alpha(\beta\alpha)^k\beta\alpha^{-1}(\alpha\beta)^{-k}\alpha^{-1}]^{-k}\alpha(\alpha\beta)^k\\
			& = \alpha[(\beta\alpha)^k\beta\alpha^{-1}(\alpha\beta)^{-k}]^{-k}(\alpha\beta)^k\\
	\end{aligned}
\end{equation}
We can substitute these three relations for the original ones. The last two relations provide us with expressions for $\gamma$ and $\delta$ in terms of $\alpha$ and $\beta$. By getting rid of $\gamma$ of the expression for $\gamma^{m}$, we are left with only one relation: 
\[(\beta\alpha)^{k}\beta = \{(\beta\alpha)^k\beta^{\omega-2k}[(\beta\alpha)^k\beta\alpha^{-1}(\alpha\beta)^{-k}]^{-\omega+k}\}^m(\alpha\beta)^{k}\alpha.\]
Thus, we have
\begin{equation}
\nonumber
	\resizebox{\hsize}{!}{$G(\omega, 1 + m\omega, 2k)=\langle \alpha,\beta|(\beta\alpha)^{k}\beta = \{(\beta\alpha)^k\beta^{\omega-2k}[(\beta\alpha)^k\beta\alpha^{-1}(\alpha\beta)^{-k}]^{-\omega+k}\}^m(\alpha\beta)^{k}\alpha\rangle.$}
\end{equation}

For the peripheral subgroup, we will compute the meridian $\mu$ and
the surface framing $s$ which is a push-off of the knot $B(\omega, 1 + m\omega, 2k)$ in $\Sigma$. It is clear that
\begin{equation}
	\nonumber
	\resizebox{\hsize}{!}{
	$\begin{aligned}
		s & = \gamma^{1 + m\omega}\delta \\
		& = \alpha\{(\beta\alpha)^k\beta^{\omega-2k}[(\beta\alpha)^k\beta\alpha^{-1}(\alpha\beta)^{-k}]^{-\omega+k}\}^{1+m\omega}[(\beta\alpha)^k\beta\alpha^{-1}(\alpha\beta)^{-k}]^{-k}(\alpha\beta)^k. \\
	\end{aligned}$}
\end{equation}
Using the relation in the expression of $G(\omega, 1 + m\omega, 2k)$, $s$ can be simplified as follows
\[s = \gamma^{1 + m\omega}\delta = \alpha(\beta\alpha)^k\beta^{\omega-2k}(\alpha\beta)^k.\]

In order to compute $\mu$, consider the meridian based at $x_0$. It can be isotoped to be sitting in the handlebody $U$, and then it goes around the right genus of the handlebody once, so is isotopic to the generator $\alpha$. Thus,
\[\mu = \alpha.\]

Finally, we note that the linking number between $B(\omega, 1 + m\omega, 2k)$ and a push-off in $\Sigma$, by the construction of the 1-bridge braid, is $(\omega-1)(1+m\omega)+2k$, which gives us
\[s = \mu^{(\omega-1)(1+m\omega)+2k}\lambda.\]
\end{proof}
\end{proposition}

\begin{figure}
\centering
\includegraphics[scale=1]{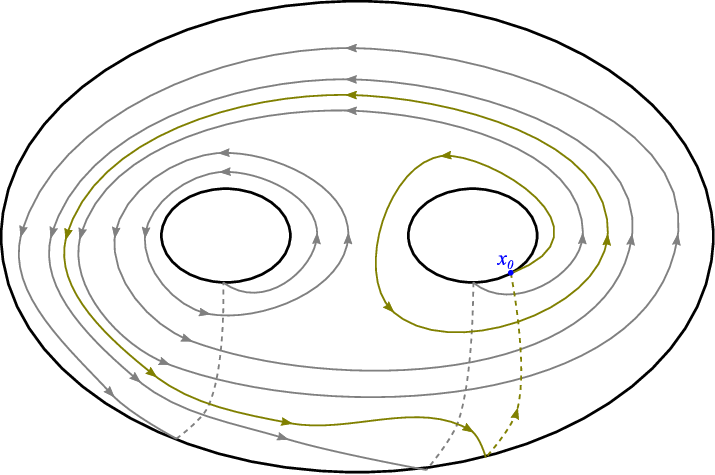}
\caption{The third generator for the fundamental group of $\Sigma \setminus \nu(B(\omega, 1 + m\omega, 2k))$.}
\label{fig: Generator3_Hmtpy_Knot_Complement_5_1_2}
\end{figure}

\begin{proposition}
\label{prop: 2n - 1}
For $B(2n + 1, 2n - 1 + m(2n + 1), 2k)$, \\
(a) The knot group is
\[G(2n + 1, 2n - 1 + m(2n + 1), 2k)\]
\[=\langle \alpha,\beta|[(\alpha\beta^{n-k+1}\alpha\beta^{-n+k})^{n+k}(\alpha\beta)^{-2k+1}\beta^{-n+k-1}]^{m+1}=\alpha\beta^{n-k+1}\alpha\beta^{-n+k}\rangle.\]
(b) The peripheral subgroup is generated by the meridian
\[\mu =\alpha\]
and the surface framing
\[s=\mu^{2n[2n - 1 + m(2n + 1)] + 2k}\lambda=\alpha\beta^{n-k+1}(\alpha\beta)^{2k-1}\alpha\beta^{n-k+1}.\]
\begin{proof}
Following the same notations as in the proof of Proposition~\ref{prop: 1}, and by following Figures~\ref{fig: Hmtpy_Knot_Complement_5_3_2} and ~\ref{fig:Generators_for_H0_and_H1}, we arrive at the relation given by the images of the green loop in Figure~\ref{fig:Generators_for_H0_and_H1}
\[\alpha\beta^{n-k+1}(\alpha\beta)^{2k-1}\alpha^{-1} = \gamma^{(m+1)(n+k)-1}.\]

\begin{figure}
\centering
\includegraphics[scale=1]{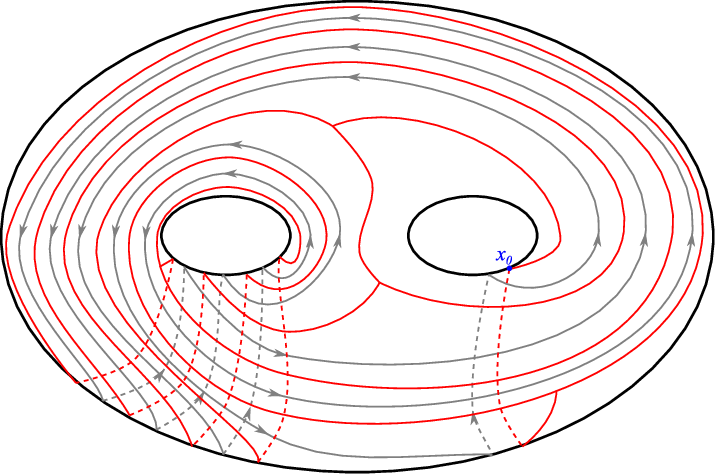}
\caption{The red graph is homotopy equivalent to the complement of $B(2n + 1, 2n - 1 + m(2n + 1), 2k)$ on the Heegaard surface $\Sigma$. Here, we show $B(5, 3, 2)$ as an example. The knot is shown in gray.}
\label{fig: Hmtpy_Knot_Complement_5_3_2}
\end{figure}
\begin{figure}
\centering
\includegraphics[scale=1]{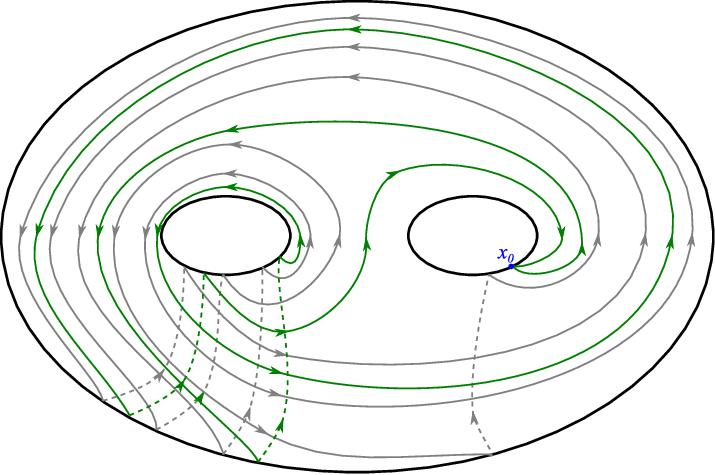}
\caption{The first generator for the fundamental group of $\Sigma \setminus \nu(B(2n + 1, 2n - 1 + m(2n + 1), 2k))$.}
\label{fig: Generator1_Hmtpy_Knot_Complement_5_3_2}
\end{figure}
\begin{figure}
\centering
\includegraphics[scale=1]{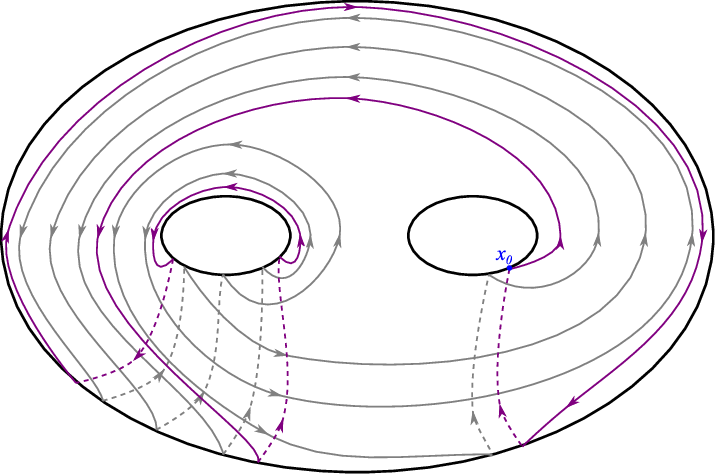}
\caption{The second generator for the fundamental group of $\Sigma \setminus \nu(B(2n + 1, 2n - 1 + m(2n + 1), 2k))$.}
\label{fig: Generator2_Hmtpy_Knot_Complement_5_3_2}
\end{figure}

Likewise, from the purple loop we get
\[\alpha\beta^{n-k}\alpha^{-1} = \gamma^{(m+1)(n-k)-1}\delta,\]
and from the olive loop we get
\[\alpha^2\beta^{n-k+1} = \gamma^{(m+1)(n-k+1)-1}\delta.\]
Using the second and the third relation, we have
\[\alpha^2\beta^{n-k+1} = \gamma^{(m+1)(n-k+1)-1}\delta=\gamma^{m+1}\gamma^{(m+1)(n-k)-1}\delta=\gamma^{m+1}\alpha\beta^{n-k}\alpha^{-1}.\]
\[\Leftrightarrow\gamma^{m+1} = \alpha^2\beta^{n-k+1}\alpha\beta^{-n+k}\alpha^{-1}.\]
This relation can replace the second relation. Moreover, it provides us with an expression for $\gamma^{m+1}$ in terms of $\alpha$ and $\beta$, so with the first relation, we have
\begin{equation}
	\nonumber
	\begin{split}
		\gamma & = \gamma^{(m+1)(n+k)}\alpha(\alpha\beta)^{-2k+1}\beta^{-n+k-1}\alpha^{-1} \\
		& = (\alpha^2\beta^{n-k+1}\alpha\beta^{-n+k}\alpha^{-1})^{n+k}\alpha(\alpha\beta)^{-2k+1}\beta^{-n+k-1}\alpha^{-1} \\
		& = \alpha(\alpha\beta^{n-k+1}\alpha\beta^{-n+k})^{n+k}(\alpha\beta)^{-2k+1}\beta^{-n+k-1}\alpha^{-1}, \\		
	\end{split}
\end{equation}
which can replace the first relation. We can also write $\delta$ in terms of $\alpha$ and $\beta$ by replacing $\gamma$ of the second relation with its expression in terms of $\alpha$ and $\beta$:
\begin{equation}
	\nonumber
	\begin{split}
		\delta & = \gamma^{-(m+1)(n-k)+1}\alpha\beta^{n-k}\alpha^{-1} \\
		& = [\alpha(\alpha\beta^{n-k+1}\alpha\beta^{-n+k})^{n+k}(\alpha\beta)^{-2k+1}\beta^{-n+k-1}\alpha^{-1}]^{-(m+1)(n-k)+1}\alpha\beta^{n-k}\alpha^{-1} \\
		& = \alpha[(\alpha\beta^{n-k+1}\alpha\beta^{-n+k})^{n+k}(\alpha\beta)^{-2k+1}\beta^{-n+k-1}]^{-(m+1)(n-k)+1}\beta^{n-k}\alpha^{-1}.\\
	\end{split}
\end{equation}
By substitution, we are left with only one relation: 
\[[(\alpha\beta^{n-k+1}\alpha\beta^{-n+k})^{n+k}(\alpha\beta)^{-2k+1}\beta^{-n+k-1}]^{m+1}=\alpha\beta^{n-k+1}\alpha\beta^{-n+k}.\]
Thus, we have
\[G(2n + 1, 2n - 1 + m(2n + 1), 2k)\]
\[=\langle \alpha,\beta|[(\alpha\beta^{n-k+1}\alpha\beta^{-n+k})^{n+k}(\alpha\beta)^{-2k+1}\beta^{-n+k-1}]^{m+1}=\alpha\beta^{n-k+1}\alpha\beta^{-n+k}\rangle.\]

For the peripheral subgroup, the surface framing $s$ is a push-off of the knot $B(2n + 1, 2n - 1 + m(2n + 1), 2k)$ along $\Sigma$, so
\[s = \gamma^{2n-1+m(2n+1)}\delta.\]
\begin{figure}
\centering
\includegraphics[scale=1]{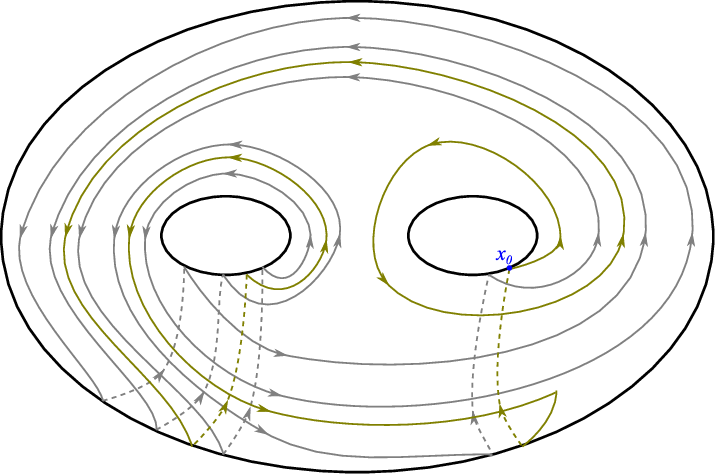}
\caption{The third generator for the fundamental group of $\Sigma \setminus \nu(B(2n + 1, 2n - 1 + m(2n + 1), 2k))$.}
\label{fig: Generator3_Hmtpy_Knot_Complement_5_3_2}
\end{figure}
We note that this is actually the product of the right sides of the first and the third relations, so we get
\[s = \alpha\beta^{n-k+1}(\alpha\beta)^{2k-1}\alpha\beta^{n-k+1}.\]

In order to compute $\mu$, for the same reason as in the proof of Proposition~\ref{prop: 1}, since the meridian based at $x_0$ can be isotoped to be going around the right genus of the handlebody $U$ once as in Figure~\ref{fig: Hmtpy_Knot_Complement_5_3_2}, it is isotopic to the generator $\alpha$. Thus,
\[\mu = \alpha.\]

Finally, we note that the linking number between $B(2n + 1, 2n - 1 + m(2n + 1), 2k)$ and a push-off along $\Sigma$, by the construction of the 1-bridge braid, is $2n[2n - 1 + m(2n + 1)] + 2k$, which gives us
\[s = \mu^{2n[2n - 1 + m(2n + 1)] + 2k}\lambda.\]
\end{proof}
\end{proposition}
\begin{figure}
\centering
\includegraphics[scale=1]{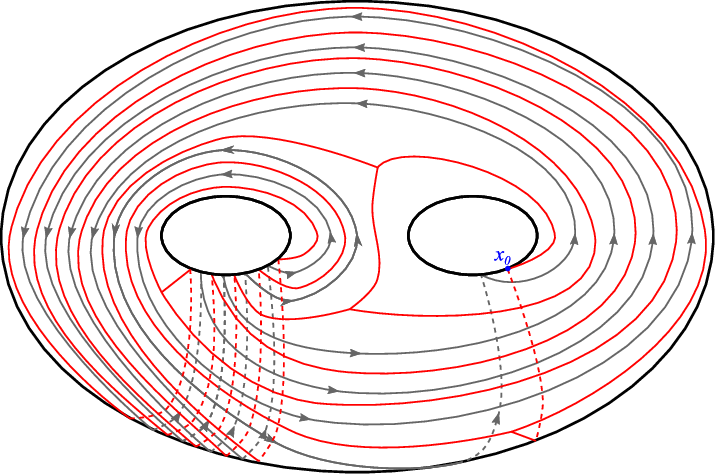}
\caption{The red graph is homotopy equivalent to the complement of $B(2n, 2n - 2 + 2mn, 2k - 1)$ on the Heegaard surface $\Sigma$. Here, we show $B(6, 4, 3)$ as an example. The knot is shown in gray.}
\label{fig: Hmtpy_Knot_Complement_6_4_3}
\end{figure}
\begin{proposition}
\label{prop: 2n - 2}
For $B(2n, 2n - 2 + 2mn, 2k - 1)$, \\
(a) The knot group is
\[G(2n, 2n - 2 + 2mn, 2k - 1)\]
\[=\langle \alpha,\beta|[(\alpha\beta^{n-k+1}\alpha\beta^{-n+k})^{n+k-1}(\alpha\beta)^{-2k+2}\beta^{-n+k-1}]^{m+1}=\alpha\beta^{n-k+1}\alpha\beta^{-n+k}\rangle.\]
(b) The peripheral subgroup is generated by the meridian
\[\mu =\alpha\]
and the surface framing
\[s=\mu^{(2n-1)(2n-2+2mn) + 2k-1}\lambda = \beta^{n-k+1}(\alpha\beta)^{2k-2}\alpha\beta^{n-k+1}\alpha.\]
\begin{proof}
Again following the same notations as in the proof of Proposition~\ref{prop: 1}, and by following Figures~\ref{fig: Hmtpy_Knot_Complement_6_4_3} and ~\ref{fig:Generators_for_H0_and_H1}, we obtain the relation for the green loop
\[\alpha\beta^{n-k+1}(\alpha\beta)^{2k-2}\alpha^{-1} = \gamma^{(m+1)(n+k-1)-1}.\]
As for the purple loop, we have
\[\alpha\beta^{n-k}\alpha^{-1} = \gamma^{(m+1)(n-k)-1}\delta,\]
and for the olive loop we have
\[\alpha^2\beta^{n-k+1} = \gamma^{(m+1)(n-k+1)-1}\delta.\]
Using the second and the third relation, we have
\[\alpha^2\beta^{n-k+1} = \gamma^{(m+1)(n-k+1)-1}\delta=\gamma^{m+1}\gamma^{(m+1)(n-k)-1}\delta=\gamma^{m+1}\alpha\beta^{n-k}\alpha^{-1}.\]
\[\Leftrightarrow\gamma^{m+1} = \alpha^2\beta^{n-k+1}\alpha\beta^{-n+k}\alpha^{-1}.\]

\begin{figure}
\centering
\includegraphics[scale=1]{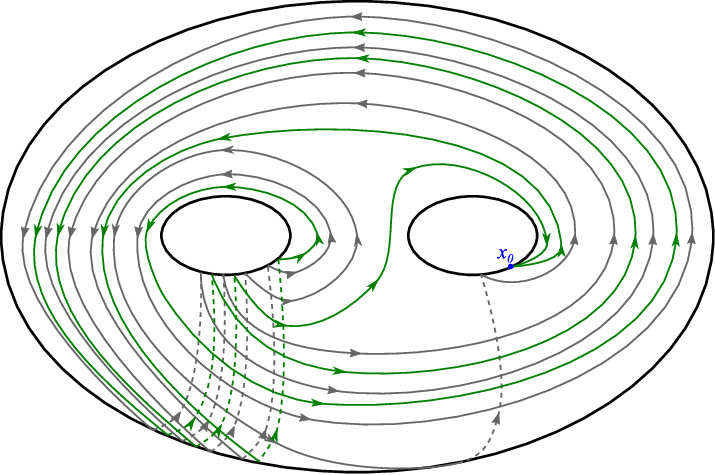}
\caption{The first generator for the fundamental group of $\Sigma \setminus \nu(B(2n, 2n - 2 + 2mn, 2k - 1))$.}
\label{fig: Generator1_Hmtpy_Knot_Complement_6_4_3}
\end{figure}
\begin{figure}
\centering
\includegraphics[scale=1]{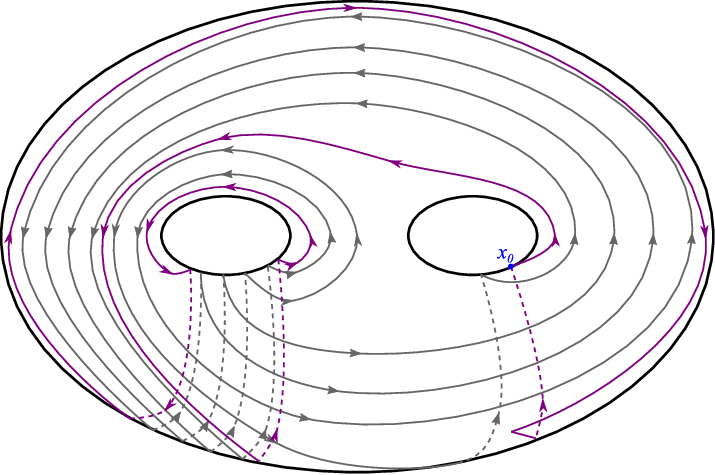}
\caption{The second generator for the fundamental group of $\Sigma \setminus \nu(B(2n, 2n - 2 + 2mn, 2k - 1))$.}
\label{fig: Generator2_Hmtpy_Knot_Complement_6_4_3}
\end{figure}

This relation can replace the second relation. Moreover, it provides us with an expression for $\gamma^{m+1}$ in terms of $\alpha$ and $\beta$, so with the first relation, we have
\begin{equation}
	\nonumber
	\begin{split}
		\gamma & = \gamma^{(m+1)(n+k-1)}\alpha(\alpha\beta)^{-2k+2}\beta^{-n+k-1}\alpha^{-1} \\
		& = (\alpha^2\beta^{n-k+1}\alpha\beta^{-n+k}\alpha^{-1})^{n+k-1}\alpha(\alpha\beta)^{-2k+2}\beta^{-n+k-1}\alpha^{-1} \\
		& = \alpha(\alpha\beta^{n-k+1}\alpha\beta^{-n+k})^{n+k-1}(\alpha\beta)^{-2k+2}\beta^{-n+k-1}\alpha^{-1}, \\
	\end{split}
\end{equation}
which can replace the first relation. We can also write $\delta$ in terms of $\alpha$ and $\beta$ by replacing $\gamma$ of the second relation with its expression in terms of $\alpha$ and $\beta$:
\begin{equation}
	\nonumber
	\begin{split}
		\delta & = \gamma^{-(m+1)(n-k)+1}\alpha\beta^{n-k}\alpha^{-1} \\
		& = [\alpha(\alpha\beta^{n-k+1}\alpha\beta^{-n+k})^{n+k-1}(\alpha\beta)^{-2k+2}\beta^{-n+k-1}\alpha^{-1}]^{-(m+1)(n-k)+1}\alpha\beta^{n-k}\alpha^{-1} \\
		& = \alpha[(\alpha\beta^{n-k+1}\alpha\beta^{-n+k})^{n+k-1}(\alpha\beta)^{-2k+2}\beta^{-n+k-1}]^{-(m+1)(n-k)+1}\beta^{n-k}\alpha^{-1}. \\
	\end{split}
\end{equation}

By substitution, we are left with only one relation: 
\[[(\alpha\beta^{n-k+1}\alpha\beta^{-n+k})^{n+k-1}(\alpha\beta)^{-2k+2}\beta^{-n+k-1}]^{m+1}=\alpha\beta^{n-k+1}\alpha\beta^{-n+k}.\]

\begin{figure}
\centering
\includegraphics[scale=1]{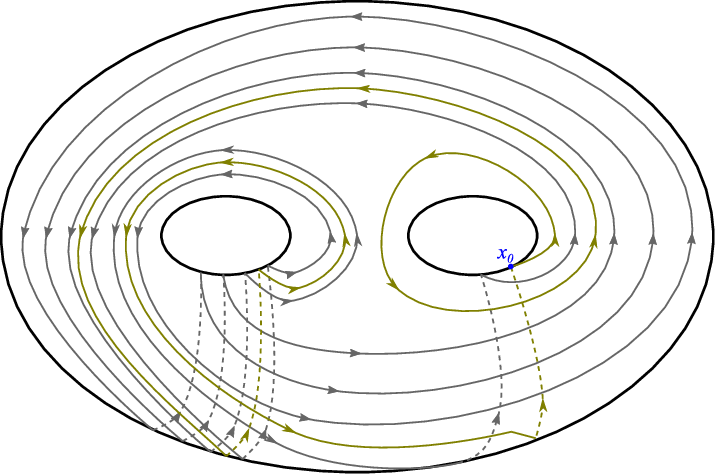}
\caption{The third generator for the fundamental group of $\Sigma \setminus \nu(B(2n, 2n - 2 + 2mn, 2k - 1))$.}
\label{fig: Generator3_Hmtpy_Knot_Complement_6_4_3}
\end{figure}

Thus, we have
\[G(2n, 2n - 2 + 2mn, 2k - 1)\]
\[=\langle \alpha,\beta|[(\alpha\beta^{n-k+1}\alpha\beta^{-n+k})^{n+k-1}(\alpha\beta)^{-2k+2}\beta^{-n+k-1}]^{m+1}=\alpha\beta^{n-k+1}\alpha\beta^{-n+k}\rangle.\]

For the peripheral subgroup,
\[s=\gamma^{2n+2mn-2}\delta,\]
which is actually the product of the right sides of the first and the third relations, so we get
\[s = \alpha\beta^{n-k+1}(\alpha\beta)^{2k-2}\alpha\beta^{n-k+1}.\]
\[\mu = \alpha.\]

Finally, we note that the linking number between $B(2n, 2n - 2 + 2mn, 2k - 1)$ and a push-off along $\Sigma$ is $(2n-1)(2n-2+2mn) + 2k-1$, i.e.,
\[s = \mu^{(2n-1)(2n-2+2mn) + 2k-1}\lambda.\]
\end{proof}
\end{proposition}

\section{Non-Left-Orderability of Certain 1-bridge Braids}
\label{sec:LO}
In this section, we prove Theorem~\ref{thm:Lia} by applying Theorem~\ref{thm:CGHVCond} on the three families of $1$-bridge braids specified in Propositions~\ref{prop: 1}, \ref{prop: 2n - 1}, and \ref{prop: 2n - 2}.

First, from Propositions~\ref{prop: 1}, ~\ref{prop: 2n - 1}, and ~\ref{prop: 2n - 2}, we note that in each of the three cases, the knot group $G(\omega, t + m\omega, b)$ can be generated by $\alpha$, $\beta$, satisfying $\mu = \alpha$, $s$ is a word that excludes $\alpha^{-1}$ and $\beta^{-1}$ and contains at least one $\alpha$. Furthermore, the framing of the longitude $s$ is always $(\omega - 1)(t + m\omega) + b > 0$. Therefore, it is sufficient to prove that for any homomorphism $\Phi:G(\omega, t + m\omega, b)\rightarrow Homeo^{+}(\mathbb{R})$, $\alpha t > t$ for all $t\in\mathbb{R}$ implies $\beta t \ge t$ for all $t\in\mathbb{R}$, where $wt$ denotes a shorthand for $\Phi(w)t$ for any word $w$.

It is a fact that if $a$ and $b$ are real numbers, $a < b$, and $w \in Homeo^+(\R)$, then $wa < wb$. We give the following two lemmas about groups acting on the real line that are needed in our proof of Theorem~\ref{thm:Lia}.
\begin{lemma}
	Let $\Phi: G \rightarrow Homeo^+(\R)$ and $\alpha \in G$. If $\alpha t > t$ for all $t \in \R$, then $w_1 w_2 t > w_1 \alpha^{-1} w_2 t$ for all $t \in \R$ and any $w_1, w_2 \in G$.
	\begin{proof}
		For any words $w_1, w_2 \in G$, $w_2 t > \alpha^{-1} w_2 t$, since $\alpha t' > t'$ for $t' = \alpha^{-1}w_2 t$. Then since $\Phi(w_1) \in Homeo^+(\R)$, applying $\Phi(w_1)$ on both sides, we have $w_1 w_2 t > w_1 \alpha^{-1} w_2 t$.
	\end{proof}
	\label{lemma: group_action1}
\end{lemma}
\begin{lemma}
	Let $\Phi: G \rightarrow Homeo^+(\R)$ and $w_i \in G, i = 1, 2, 3, 4$. If $w_1 t > w_2 t$ for all $t \in \R$ and $w_3 t > w_4 t$ for all $t \in \R$, then $w_1 w_3 t > w_2 w_4 t$ for all $t \in \R$.
	\begin{proof}
		First we have $w_1 w_3 t > w_2 w_3 t$, since $w_1 t' > w_2 t'$ for $t' = w_3 t$. Since $\Phi(w_2) \in Homeo^+(\R)$, we get $w_2 w_3 t > w_2 w_4 t$. Thus $w_1 w_3 t > w_2 w_3 t > w_2 w_4 t$ for all $t \in \R$.
	\end{proof}
	\label{lemma: group_action2}
\end{lemma}
	By induction, we have the following Corollary.
\begin{corollary}
	Let $\Phi: G \rightarrow Homeo^+(\R)$ and $w_i, w_i' \in G, i = 1, 2,\cdots, n$. If $w_i t > w_i' t$ for all $i$ and all $t \in \R$, then $w_1 w_2\cdots w_n t > w_1' w_2'\cdots w_n' t$ for all $t \in \R$.
	\label{cor: group_action3}
\end{corollary}
\begin{lemma}
	For any homomorphism $\Phi: G(\omega, 1 + m\omega, 2k)\rightarrow Homeo^{+}(\R)$, $1\le k\le [(\omega-2)/2]$, $m \ge 0$, $\alpha t > t$ for all $t\in\R$ implies $\beta t \ge t$ for all $t\in\R$.
\begin{proof}
Recall the expression for the knot group:
\begin{equation}
	\nonumber	
	\resizebox{\hsize}{!}{$G(\omega, 1 + m\omega, 2k) = \langle \alpha,\beta|(\beta\alpha)^{k}\beta = \{(\beta\alpha)^k\beta^{\omega-2k}[(\beta\alpha)^k\beta\alpha^{-1}(\alpha\beta)^{-k}]^{-\omega+k}\}^m(\alpha\beta)^{k}\alpha\rangle.$}
\end{equation}
The relation can be rewritten slightly as:
\[1 = \{(\beta\alpha)^k\beta^{\omega-2k}[(\alpha\beta)^k\alpha(\alpha\beta)^{-k}\beta^{-1}]^{\omega-k}\}^m(\alpha\beta)^k\alpha(\alpha\beta)^{-k}\beta^{-1}.\]
Applying both sides of the relation on $t$, we have
\[1t = [(\beta\alpha)^k\beta^{\omega-2k}C_0^{\omega-k}]^mC_0 t,\]
where $C_0 = (\alpha\beta)^k\alpha(\alpha\beta)^{-k}\beta^{-1}.$
Assume $\alpha t > t$ for all $t\in\R$. By Lemma~\ref{lemma: group_action1}, since $\alpha t > t$ for all $t\in\R$,
	\[C_0t = [(\alpha\beta)^k\alpha][(\alpha\beta)^{-k}\beta^{-1}]t > [(\alpha\beta)^k\alpha]\alpha^{-1}[(\alpha\beta)^{-k}\beta^{-1}]t = \beta^{-1}t,\quad \forall t\in\R.\]
	From $\alpha t > t, \forall t$, we know $\beta\alpha t > \beta t, \forall t$. Then, by Corollary~\ref{cor: group_action3} with all $w_i = \beta\alpha$ and $w_i' = \beta$, we get $(\beta\alpha)^k t> \beta^k t, \forall t$.
Note that $[(\beta\alpha)^k\beta^{\omega-2k}C_0^{\omega-k}]^mC_0$ is a product of non-negative powers of $(\beta\alpha)^k$, $\beta^{\omega-2k}$, and positive powers of $C_0$. By Corollary~\ref{cor: group_action3} again, we have
\[1t > (\beta^k\beta^{\omega-2k}\beta^{-\omega+k})^m\beta^{-1}t = \beta^{-1}t.\]
Since $\beta$ is order-preserving, applying it on both sides, we get
\[\beta t > t\]
for all $t\in \R$.
\end{proof}
\label{lemma: left_order_1}
\end{lemma}

\begin{remark}
If $m = 0$, by applying Markov's Theorem and the Type II Markov Move, $B(\omega, 1, 2k)$ is actually the torus knot $T(2, 2k + 1)$. The fact the conjecture is true for torus knots follows from a result of Moser \cite{Moser}, which shows that surgery along a torus knot is always either a lens space, a connected sum of lens spaces, or Seifert fibered. A result of Boyer, Gordon, and Watson \cite[Theorem 4]{BGWconj} shows that the conjecture is true for Seifert fibered spaces.
\label{remark: r_1}
\end{remark}

\begin{lemma}
For any homomorphism $\Phi: G(2n + 1, 2n - 1 + m(2n + 1), 2k, r)\rightarrow Homeo^{+}(\R)$, $1\le k\le n - 1$, $m \ge 0$, $\alpha t > t$ for all $t\in\R$ implies $\beta t \ge t$ for all $t\in\R$.

For any homomorphism $\Phi: G(2n, 2n - 2 + 2mn, 2k - 1, r)\rightarrow Homeo^{+}(\R)$, $1\le k\le n - 1$, $m \ge 0$, $\alpha t > t$ for all $t\in\R$ implies $\beta t\ge t$ for all $t\in\R$.
\begin{proof}
We note that the two knot groups can be written in a unified form:
\[G(\omega, t + m\omega, b) = \langle \alpha,\beta|[(C_1C_2)^{n-k+b}(\alpha\beta)^{-b+1}\beta^{-n+k-1}]^{m+1} = C_1C_2\rangle,\]
where $C_1 = \alpha\beta^{n-k+1}$ and $C_2 = \alpha\beta^{-n+k}$, and in the first case, $\omega = 2n + 1$, $t = 2n - 1$, $b = 2k$, while $\omega = 2n$, $t = 2n - 2$, $b = 2k - 1$ in the second.
We can rewrite the left side in a slightly different way:
\[(C_1C_2)^{n-k+b}[(\alpha\beta)^{-b+1}\beta^{-n+k-1}(C_1C_2)^{n-k+b}]^{m}(\alpha\beta)^{-b+1}\beta^{-n+k-1} = C_1C_2.\]
\[\Rightarrow(C_1C_2)^{n-k+b-1}[(\alpha\beta)^{-b+1}\beta^{-n+k-1}(C_1C_2)^{n-k+b}]^{m} = \beta^{n-k+1}(\alpha\beta)^{b-1}.\]
For any $t_0\in\R$, applying both sides of the relation on $t_0$, we have
\[\beta^{n-k+1}(\alpha\beta)^{b-1}t_0 = (C_1C_2)^{n-k+b-1}[(\alpha\beta)^{-b+1}\beta^{-n+k-1}(C_1C_2)^{n-k+b}]^{m}t_0.\]
Assume $\alpha t > t$ for all $t\in\R$. By Lemma~\ref{lemma: group_action1} and Corollary~\ref{cor: group_action3}, we can add $\alpha^{-1}$ on one side of the equation to get a strict inequality on any $t_0$. Thus,
\[ (C_1C_2)^{n-k+b-1}t_0 = (C_1C_2)^{n-k}(C_1C_2)^{b-1}t_0\]
\[> (\alpha^{-1}C_1\alpha^{-1}C_2)^{n-k}(C_1\alpha^{-1}C_2)^{b-1}t_0 = \beta^{n-k}(\alpha\beta)^{b-1}t_0\]
for all $t_0\in\R$. Similarly,
\[ (C_1C_2)^{n-k+b}t_0 > \beta^{n-k+1}(\alpha\beta)^{b-1}t_0\]
for all $t_0\in\R$. Thus,
\[(C_1C_2)^{n-k+b-1}[(\alpha\beta)^{-b+1}\beta^{-n+k-1}(C_1C_2)^{n-k+b}]^{m}t_0\]
\[> \beta^{n-k}(\alpha\beta)^{b-1}[(\alpha\beta)^{-b+1}\beta^{-n+k-1}\beta^{n-k+1}(\alpha\beta)^{b-1}]t_0 = \beta^{n-k}(\alpha\beta)^{b-1}t_0.\]
\[\Rightarrow\beta^{n-k+1}(\alpha\beta)^{b-1}t_0 > \beta^{n-k}(\alpha\beta)^{b-1}t_0,\]
for all $t_0\in\R$.
Given any $t\in\R$, let $t_0 = (\alpha\beta)^{-b+1}\beta^{-n+k}t$. We get
\[\beta t > t.\]
	\end{proof}
\label{lemma: left_order_2}
\end{lemma}
Now we are ready to prove Theorem~\ref{thm:Lia}.
\begin{proof}[Proof of Theorem~\ref{thm:Lia}]
	Given the results in Lemmas~\ref{lemma: left_order_1} and \ref{lemma: left_order_2}, the conditions in Theorem~\ref{thm:CGHVCond} are satisfied. Therefore, we have proved Theorem~\ref{thm:Lia}.
\end{proof}
\bibliographystyle{amsalpha}

\bibliography{mybib}

\end{document}